\documentclass[11pt]{article}

\usepackage[T1]{fontenc}
\usepackage[utf8]{inputenc}
\usepackage{lmodern}
\usepackage{microtype}
\usepackage{indentfirst}
\usepackage[a4paper,margin=1in]{geometry}
\usepackage{amsmath,amssymb,amsthm}
\usepackage{empheq}
\usepackage{aliascnt}
\usepackage{needspace}
\usepackage{xcolor}
\usepackage{hyperref}
\usepackage[nameinlink,capitalise,noabbrev]{cleveref}

\hypersetup{
  colorlinks=true,
  linkcolor=blue,
  citecolor=blue,
  urlcolor=blue,
  pdftitle={A sharp weighted commutator inequality for rank-one matrices},
  pdfauthor={Zhekai Pang}
}

\makeatletter
\renewcommand{\tagform@}[1]{%
  \maketag@@@{\color{blue}(\ignorespaces#1\unskip\@@italiccorr)}%
}
\makeatother

\newtheorem{theorem}{Theorem}[section]
\newaliascnt{lemma}{theorem}
\newtheorem{lemma}[lemma]{Lemma}
\aliascntresetthe{lemma}
\newaliascnt{corollary}{theorem}
\newtheorem{corollary}[corollary]{Corollary}
\aliascntresetthe{corollary}
\newaliascnt{proposition}{theorem}
\newtheorem{proposition}[proposition]{Proposition}
\aliascntresetthe{proposition}
\theoremstyle{remark}
\newaliascnt{remark}{theorem}

\aliascntresetthe{remark}

\newcommand{\C}{\mathbb C}
\newcommand{\R}{\mathbb R}
\newcommand{\F}{\mathrm F}
\newcommand{\norm}[1]{\lVert#1\rVert}
\newcommand{\abs}[1]{\lvert#1\rvert}
\newcommand{\CalC}{\mathcal C}
\DeclareMathOperator{\tr}{tr}
\DeclareMathOperator{\rank}{rank}
\DeclareMathOperator{\spanof}{span}
\DeclareMathOperator{\diag}{diag}

\title{A sharp weighted commutator inequality
for rank-one matrices}
\author{Zhekai Pang\\
Universitat Pompeu Fabra\\
\texttt{zhekai.pang@upf.edu}}
\date{}

\begin{document}
\maketitle

\begin{abstract}
Let $A\in\C^{n\times n}$ be positive definite Hermitian, where $n\ge2$,
and define $\norm{X}_A^2=\tr(AX^*X)$ for all $X\in\C^{n\times n}$.
We determine the smallest constant $\CalC_A$ for which
$$
 \norm{XY-YX}_{\F}^{2}
 \le \CalC_A\norm{X}_A^2\norm{Y}_A^2
$$
holds for all rank-one $X,Y\in\C^{n\times n}$, where $\norm{\cdot}_{\F}$
denotes the Frobenius norm. The constant depends only on the two smallest eigenvalues of $A$ and
is determined by a strictly monotone scalar equation. This answers the
sharp-constant question for rank-one matrices formulated by Fang and
Cheng. We characterize all equality pairs. Also, we present an
application to density matrices in quantum physics.
\end{abstract}



\section{Introduction}

Let $M_n(\C)$ denote the space of $n\times n$ complex matrices, where
$n\ge2$. The B\"ottcher--Wenzel inequality~\cite{BW,VJ} states that
\[
 \norm{XY-YX}_{\F}^{2}
 \le 2\norm{X}_{\F}^{2}\norm{Y}_{\F}^{2}
\]
for all $X,Y\in M_n(\C)$, where $\norm{X}_{\F}^{2}=\tr(X^*X)$ denotes
the squared Frobenius norm.
Mayumi, Kimura, Ohno, and Chru\'sci\'nski~\cite{MKOC} introduced
weighted versions using
\[
 \norm{X}_A^2=\tr(AX^*X),
\]
where $A\in M_n(\C)$ is positive definite Hermitian and $X^*$ denotes
the conjugate transpose of $X\in M_n(\C)$. Throughout the paper, the
eigenvalues of $A$ are ordered as
\[
 0<\lambda_1\le\lambda_2\le\cdots\le\lambda_n,
\]
with multiplicity. In particular, $\lambda_1$ and $\lambda_2$ are its
two smallest eigenvalues.

Fang and Cheng~\cite[Theorem~2.4]{FC} proved
\begin{equation}\label{eq:known-fully-weighted}
 \norm{XY-YX}_A^2
 \le \left(\frac{1}{\lambda_1}+\frac{1}{\lambda_2}\right)
       \norm{X}_A^2\norm{Y}_A^2
\end{equation}
for all rank-one $X,Y\in M_n(\C)$, with the best possible constant.
The inequality is also known when one factor is normal, by Kimura, Mayumi, Ohno, Lee,
and Chru\'sci\'nski~\cite[Supplementary Information, Theorem~1]{KMOLC}.

We consider the different problem with an unweighted output norm,
whose sharp constant is
\begin{equation}\label{eq:C-def}
 \CalC_A
 =\max_{\substack{X,Y\in M_n(\C)\\\rank X=\rank Y=1}}
   \frac{\norm{XY-YX}_{\F}^{2}}
        {\norm{X}_A^2\norm{Y}_A^2}.
\end{equation}
Fang and Cheng~\cite[Section~4]{FC} obtained the upper bound
$\lambda_1^{-1}(\lambda_1^{-1}+\lambda_2^{-1})$ for $\CalC_A$.
They showed that it is sharp when $\lambda_1=\lambda_2$ but not when
$\lambda_1<\lambda_2$, and observed that the problem reduces to dimension
two. They left the exact constant in the latter case undetermined.
For unweighted input norms, the singular-value region was studied by
Wenzel~\cite{Wenzel}, determined in dimension two by Cheng and
Liang~\cite{CL}, and determined in arbitrary dimension by Cheng and
Jiao~\cite{CJ}. Further sharp bounds appear in~\cite{CAJ}.

Our main result determines~\eqref{eq:C-def}.

\begin{theorem}\label{thm:main}
Let $A\in M_n(\C)$ be positive definite Hermitian, and let
$\lambda_1\le\lambda_2$ be its two smallest eigenvalues, counted with
multiplicity. If $\lambda_1=\lambda_2$, then $\CalC_A=2/\lambda_1^2$.
If $\lambda_1<\lambda_2$, there is a unique pair of real numbers
$0<x_0<t_0<1$ satisfying
\begin{equation}\label{eq:optimal-system}
 \begin{cases}
  (t_0-x_0)(1-t_0^2)=4t_0x_0^2(1+x_0t_0),\\[5pt]
  \displaystyle\frac{t_0^2-x_0^2}{(t_0+3x_0)(1+x_0t_0)}
  =\frac{\lambda_2-\lambda_1}{\lambda_2+\lambda_1}.
 \end{cases}
\end{equation}
For this pair, $t_0>{(\lambda_2-\lambda_1)/(\lambda_2+\lambda_1)}$ and
\begin{equation}\label{eq:C-formula}
 \CalC_A
 =\frac{8(1+x_0t_0)^3}
 {\bigl[\lambda_1(1+t_0)+\lambda_2(1-t_0)\bigr]^2(1+t_0^2+2x_0t_0)}.
\end{equation}
Thus the constant depends only on $\lambda_1$ and $\lambda_2$.
For real symmetric $A$, it is also the sharp constant for all
real rank-one $X,Y\in M_n(\R)$.
\end{theorem}

Sections~\ref{sec:numerator} and~\ref{sec:weight} give sharp bounds for
the numerator and denominator in~\eqref{eq:C-def}.
Section~\ref{sec:scalar} determines the sharp constant.
Section~\ref{sec:equality} describes all equality cases and gives
consequences of the main result. Section~\ref{sec:density} presents an
application to density matrices in quantum physics.

\section{A sharp unweighted estimate}\label{sec:numerator}

For rank-one $X,Y\in M_n(\C)$ of unit Frobenius norm, write $X=ab^*$
and $Y=cd^*$ with unit vectors $a,b,c,d\in\C^n$.
Below, $\norm{\cdot}$ is the Euclidean norm and $x,t$ are arbitrary
parameters, distinct from the optimal pair $x_0,t_0$ in~\cref{thm:main}.

\begin{lemma}\label{lem:numerator}
Let $a,b,c,d\in\C^n$ be unit vectors, and write
$x=\abs{a^*c}$ and $t=\abs{b^*d}$. Then
\begin{equation}\label{eq:f-bound}
 \norm{ab^*cd^*-cd^*ab^*}_{\F}^{2}
 \le f(x,t)
 :=(1+xt)\left[1-xt+\sqrt{(1-x^2)(1-t^2)}\right].
\end{equation}
For every pair of unit vectors $b,d\in\C^n$ and every $x\in[0,1]$,
there are unit vectors $a,c\in\C^n$ with $\abs{a^*c}=x$ that attain
this bound.

If $0<x,t<1$, $b^*d=t$, and $a^*c=-x$, then there exist orthonormal
pairs $e_1,e_2\in\C^n$ and $u_1,u_2\in\C^n$ such that
\begin{align}
 b&=\sqrt{\frac{1+t}{2}}\,e_1+\sqrt{\frac{1-t}{2}}\,e_2,
 &d&=\sqrt{\frac{1+t}{2}}\,e_1-\sqrt{\frac{1-t}{2}}\,e_2,
 \label{eq:right-vectors}\\
 a&=\sqrt{\frac{1-x}{2}}\,u_1-\sqrt{\frac{1+x}{2}}\,u_2,
 &c&=\sqrt{\frac{1-x}{2}}\,u_1+\sqrt{\frac{1+x}{2}}\,u_2.
 \label{eq:left-vectors}
\end{align}
Under these assumptions, equality in~\eqref{eq:f-bound} holds if and
only if $u_1=ze_1$ and $u_2=ze_2$ for some $z\in\C$ with $\abs{z}=1$.
\end{lemma}

\begin{proof}
Multiplying $c$ and $d$ independently by scalars of modulus one changes
$cd^*$ only by a scalar of modulus one, and therefore preserves the
Frobenius norm of the commutator. Since $x=\abs{a^*c}$ and
$t=\abs{b^*d}$ are also unchanged, we may assume $a^*c=-x$ and $b^*d=t$.
Since the four vectors have unit length, $\norm{b+d}^2=2(1+t)$,
$\norm{b-d}^2=2(1-t)$, $\norm{a+c}^2=2(1-x)$,
$\norm{c-a}^2=2(1+x)$, $(b+d)^*(b-d)=0$, and $(a+c)^*(c-a)=0$.

If $t<1$, both $b+d$ and $b-d$ are nonzero. We may therefore define
\begin{equation}\label{eq:right-basis}
 e_1=\frac{b+d}{\sqrt{2(1+t)}},\qquad
 e_2=\frac{b-d}{\sqrt{2(1-t)}}.
\end{equation}
The norm and inner-product identities above show that $e_1,e_2$ are
orthonormal. If $t=1$, then $b=d$. In this case take $e_1=b=d$
and choose any unit vector $e_2\perp e_1$ instead.

If $x<1$, both $a+c$ and $c-a$ are nonzero. Define
\[
 u_1=\frac{a+c}{\sqrt{2(1-x)}},\qquad
 u_2=\frac{c-a}{\sqrt{2(1+x)}}.
\]
These vectors are orthonormal by the same calculations.
If $x=1$, then $c=-a$. In this case take $u_2=c=-a$ and choose
any unit vector $u_1\perp u_2$. The required orthogonal unit vectors
exist because $n\ge2$.

To recover $b,d$ and $a,c$, add and subtract their sums and differences:
\begin{align*}
 b&=\frac{(b+d)+(b-d)}{2}
   =\sqrt{\frac{1+t}{2}}\,e_1+\sqrt{\frac{1-t}{2}}\,e_2,\\
 d&=\frac{(b+d)-(b-d)}{2}
   =\sqrt{\frac{1+t}{2}}\,e_1-\sqrt{\frac{1-t}{2}}\,e_2,\\
 a&=\frac{(a+c)-(c-a)}{2}
   =\sqrt{\frac{1-x}{2}}\,u_1-\sqrt{\frac{1+x}{2}}\,u_2,\\
 c&=\frac{(a+c)+(c-a)}{2}
   =\sqrt{\frac{1-x}{2}}\,u_1+\sqrt{\frac{1+x}{2}}\,u_2.
\end{align*}
When $t=1$, the coefficient of $e_2$ is zero. When $x=1$, the
coefficient of $u_1$ is zero. Thus these identities prove
\eqref{eq:right-vectors} and~\eqref{eq:left-vectors} for all
$x,t\in[0,1]$, including the endpoints.

We next compare the two orthonormal pairs. Define
\[
 K=\begin{pmatrix}e_1^*u_1&e_1^*u_2\\e_2^*u_1&e_2^*u_2\end{pmatrix}.
\]
For any $v=(v_1,v_2)^\top\in\C^2$, let $w=v_1u_1+v_2u_2$.
Because $u_1,u_2$ are orthonormal,
$\norm{w}^2=\abs{v_1}^2+\abs{v_2}^2=\norm{v}^2$.
The definition of $K$ gives $Kv=(e_1^*w,e_2^*w)^\top$.
The vector $w-(e_1^*w)e_1-(e_2^*w)e_2$ is orthogonal to both
$e_1$ and $e_2$. Applying the Pythagorean identity to these three
orthogonal components of $w$, we obtain
\[
 \norm{Kv}^2=\abs{e_1^*w}^2+\abs{e_2^*w}^2
 =\norm{w}^2-\norm{w-(e_1^*w)e_1-(e_2^*w)e_2}^2\le\norm{v}^2.
\]
Thus $\norm{Kv}\le\norm{v}$ for every $v\in\C^2$, so $K$ is a contraction.

\Needspace{18\baselineskip}
\begin{align}
 \norm{ab^*cd^*-cd^*ab^*}_{\F}^{2}
 &=\norm{(b^*c)ad^*-(d^*a)cb^*}_{\F}^{2}\notag\\
 &=\abs{b^*c}^2\norm{ad^*}_{\F}^{2}
   +\abs{d^*a}^2\norm{cb^*}_{\F}^{2}
   -2\operatorname{Re}\bigl[\overline{b^*c}\,(d^*a)\tr(da^*cb^*)\bigr]\notag\\
 &=\abs{b^*c}^2+\abs{d^*a}^2
   +2xt\operatorname{Re}\bigl(\overline{b^*c}\,d^*a\bigr)\notag\\
 &=\frac{1}{2}(1+xt)\abs{b^*c+d^*a}^{2}
   +\frac{1}{2}(1-xt)\abs{b^*c-d^*a}^{2}\notag\\
 &=\frac{1}{2}(1+xt)
   \abs{\sqrt{(1+t)(1-x)}\,e_1^*u_1
          +\sqrt{(1-t)(1+x)}\,e_2^*u_2}^{2}\notag\\
 &\quad+\frac{1}{2}(1-xt)
   \abs{\sqrt{(1+t)(1+x)}\,e_1^*u_2
          +\sqrt{(1-t)(1-x)}\,e_2^*u_1}^{2}\notag\\
 &=:\Phi(K),
 \label{eq:quadratic-form}
\end{align}
where the third equality uses $\norm{ad^*}_{\F}=\norm{cb^*}_{\F}=1$ and
$\tr(da^*cb^*)=(a^*c)(b^*d)=-xt$, and the fifth equality follows by
substituting~\eqref{eq:right-vectors} and~\eqref{eq:left-vectors}.

Viewed as a function of the four entries of $K$, $\Phi$ is a sum of
squared absolute values of linear functions with nonnegative coefficients.
Thus $\Phi$ is convex on $M_2(\C)$. As shown in~\cite{RD},
every contraction in $M_2(\C)$ is a convex combination of unitary matrices.
Write $K=\sum_j p_jU_j$ as such a finite combination, where each $U_j$
is unitary, $p_j\ge0$, and $\sum_jp_j=1$. Thus
\[
 \Phi(K)=\Phi\Bigl(\sum_jp_jU_j\Bigr)
 \le\sum_jp_j\Phi(U_j)\le\max_j\Phi(U_j).
\]
Thus a bound for all unitary matrices also holds for every contraction.

Suppose now that $K$ is unitary. Since every row and column has squared
norm one,
\[
 \abs{e_1^*u_1}=\abs{e_2^*u_2},\qquad
 \abs{e_1^*u_2}=\abs{e_2^*u_1},\qquad
 \abs{e_1^*u_1}^2+\abs{e_1^*u_2}^2=1.
\]
Applying the triangle inequality to the two absolute values
in~\eqref{eq:quadratic-form}, we obtain
\begin{align*}
 \Phi(K)
 &\le\frac{1}{2}(1+xt)
 \left[\sqrt{(1+t)(1-x)}\,\abs{e_1^*u_1}
       +\sqrt{(1-t)(1+x)}\,\abs{e_2^*u_2}\right]^2\\
 &\quad+\frac{1}{2}(1-xt)
 \left[\sqrt{(1+t)(1+x)}\,\abs{e_1^*u_2}
       +\sqrt{(1-t)(1-x)}\,\abs{e_2^*u_1}\right]^2\\
 &=\frac{1}{2}(1+xt)
 \left[\sqrt{(1+t)(1-x)}+\sqrt{(1-t)(1+x)}\right]^2\abs{e_1^*u_1}^2\\
 &\quad+\frac{1}{2}(1-xt)
 \left[\sqrt{(1+t)(1+x)}+\sqrt{(1-t)(1-x)}\right]^2\abs{e_1^*u_2}^2\\
 &=(1+xt)\bigl[1-xt+\sqrt{1-x^2}\sqrt{1-t^2}\bigr]\abs{e_1^*u_1}^2
   +(1-xt)\bigl[1+xt+\sqrt{1-x^2}\sqrt{1-t^2}\bigr]\abs{e_1^*u_2}^2\\
 &\le(1+xt)\bigl[1-xt+\sqrt{1-x^2}\sqrt{1-t^2}\bigr]\abs{e_1^*u_1}^2
   +(1+xt)\bigl[1-xt+\sqrt{1-x^2}\sqrt{1-t^2}\bigr]\abs{e_1^*u_2}^2\\
 &=(1+xt)\bigl[1-xt+\sqrt{1-x^2}\sqrt{1-t^2}\bigr]\\
 &=f(x,t).
\end{align*}
By the preceding convexity argument, this proves~\eqref{eq:f-bound}
for every contraction $K$.

For attainment, fix unit vectors $b,d$ and $x\in[0,1]$, with
$t=\abs{b^*d}$. Let $z=\overline{b^*d}/t$ if $t>0$, and $z=1$ if
$t=0$. Then $\abs{z}=1$ and $b^*(zd)=t$. Choose $e_1,e_2$ so that
\eqref{eq:right-vectors} holds with $d$ replaced by $zd$.
Take $u_1=e_1$, $u_2=e_2$, and define $a,c$ by~\eqref{eq:left-vectors}.
The vectors $a$ and $c$ have unit length and satisfy $a^*c=-x$.
Now $K=I_2$,
where $I_2$ is the $2\times2$ identity matrix, so~\eqref{eq:quadratic-form} gives
\begin{align*}
 f(x,t)&=\Phi(I_2)=\norm{ab^*c(zd)^*-c(zd)^*ab^*}_{\F}^{2}\\
 &=\norm{\overline z\,(ab^*cd^*-cd^*ab^*)}_{\F}^{2}
  =\norm{ab^*cd^*-cd^*ab^*}_{\F}^{2}.
\end{align*}

For the equality characterization, assume $0<x,t<1$, $b^*d=t$,
and $a^*c=-x$. For a unitary $K$, equality in the last inequality of
the preceding estimate requires $e_1^*u_2=0$. The equal moduli above
then give $e_2^*u_1=0$. Since $\abs{e_1^*u_1}=\abs{e_2^*u_2}=1$, equality in the triangle
inequality holds exactly when $e_1^*u_1=e_2^*u_2=z$ for some
$z\in\C$ with $\abs{z}=1$.
Thus the unitary maximizers are exactly $zI_2$ with $\abs{z}=1$.

If a contraction $K$ attains the bound, every unitary occurring with
positive coefficient in its convex decomposition must also attain it.
Otherwise convexity would give a strict inequality for $\Phi(K)$.
Each such unitary is a scalar multiple of $I_2$, so $K=zI_2$ with
$\abs{z}\le1$. Since $\Phi(zI_2)=\abs{z}^2f(x,t)$ and $f(x,t)>0$,
equality forces $\abs{z}=1$.
The identities $e_1^*u_1=e_2^*u_2=z$ then give equality in the
Cauchy--Schwarz inequality for each pair of unit vectors. Consequently
$u_1=ze_1$ and $u_2=ze_2$.
Conversely, these relations give $K=zI_2$ and equality
in~\eqref{eq:quadratic-form}.
\end{proof}

For fixed unit vectors $b,d$, we next maximize over $a,c$.
Since the sharp bound depends on $a,c$ only through $x=\abs{a^*c}$,
define
\begin{equation}\label{eq:F-def}
 F(t)=\max_{0\le x\le1}f(x,t),\qquad 0\le t\le1.
\end{equation}
By~\cref{lem:numerator}, $F(t)$ is the largest squared commutator
norm for any fixed unit vectors $b,d$ with $\abs{b^*d}=t$.

\begin{lemma}\label{lem:x-optimum}
The endpoint values are $F(0)=2$ and $F(1)=1$.
For every fixed $t\in(0,1)$, the function $x\mapsto f(x,t)$ has a
unique maximizer on $[0,1]$, denoted by $x_*(t)$. This maximizer lies
in $(0,t)$ and is the unique solution in that interval of
\begin{equation}\label{eq:x-cubic}
 (t-x)(1-t^2)=4tx^2(1+xt).
\end{equation}
Moreover, the function $t\mapsto x_*(t)$ is smooth on $(0,1)$.
\end{lemma}

\begin{proof}
Since $f$ is continuous on $[0,1]^2$, its maximum over the fixed
interval $0\le x\le1$ is continuous in $t$.
At the endpoints,
\[
 f(x,0)=1+\sqrt{1-x^2},\qquad f(x,1)=1-x^2,
\]
so both maxima occur at $x=0$ and have the stated values.

Fix $t\in(0,1)$. For $0\le x<1$, differentiation gives
\begin{align}
 f_x(x,t)
 &=-2xt^2+\frac{\sqrt{1-t^2}\,(t-x-2tx^2)}{\sqrt{1-x^2}},
 \label{eq:fx}\\
 f_{xx}(x,t)
 &=-2t^2-\frac{\sqrt{1-t^2}\,\bigl[1+tx(3-2x^2)\bigr]}{(1-x^2)^{3/2}}<0.
 \notag
\end{align}
Thus $f_x(\cdot,t)$ is strictly decreasing. Since
\[
 f_x(0,t)=t\sqrt{1-t^2}>0,\qquad f_x(t,t)=-4t^3<0,
\]
it has exactly one zero in $(0,t)$. Denote this zero by $x_*(t)$.
The derivative is positive for $0\le x<x_*(t)$ and negative for
$x_*(t)<x<1$. Hence $x_*(t)$ is the unique maximizer on $[0,1]$,
and $0<x_*(t)<t$.

At this point, $f_x(x_*(t),t)=0$, so~\eqref{eq:fx} gives
\[
 \sqrt{1-t^2}\,\bigl[t-x_*(t)-2t x_*(t)^2\bigr]
 =2t^2x_*(t)\sqrt{1-x_*(t)^2}.
\]
Squaring and factoring, we obtain
\begin{align*}
 0&=(1-t^2)\bigl[t-x_*(t)-2t x_*(t)^2\bigr]^2
       -4t^4x_*(t)^2\bigl[1-x_*(t)^2\bigr]\\
  &=\bigl[x_*(t)-t\bigr]\Bigl(4t x_*(t)^2\bigl[1+t x_*(t)\bigr]
       -\bigl[t-x_*(t)\bigr](1-t^2)\Bigr).
\end{align*}
Since $x_*(t)<t$ by the preceding derivative test, the factor $x_*(t)-t$
is nonzero. Therefore $x_*(t)$ satisfies~\eqref{eq:x-cubic}.
To exclude additional solutions in $(0,t)$, differentiate the
right-hand side minus the left-hand side of~\eqref{eq:x-cubic}:
\[
 \frac{\partial}{\partial x}
 \bigl[4tx^2(1+xt)-(t-x)(1-t^2)\bigr]
 =12t^2x^2+8tx+1-t^2>0,\qquad 0<x<t.
\]
This difference is strictly increasing, so $x_*(t)$ is its only zero
in $(0,t)$. Its $x$-derivative is nonzero at this root. The implicit
function theorem therefore shows that $t\mapsto x_*(t)$ is smooth
on $(0,1)$.
\end{proof}

\section{The product of two Rayleigh quotients}\label{sec:weight}

For all unit vectors $a,b\in\C^n$,
\[
 \norm{ab^*}_A^2=\tr\bigl(Aba^*ab^*\bigr)=\tr(Abb^*)=b^*Ab.
\]
Thus, for $X=ab^*$ and $Y=cd^*$ with unit vectors $a,b,c,d$, the
denominator in~\eqref{eq:C-def} is $(b^*Ab)(d^*Ad)$ and is independent
of $a,c$. For each $t\in[0,1]$, define
\[
 D(t)=\min_{\substack{b,d\in\C^n\\\norm{b}=\norm{d}=1\\\abs{b^*d}=t}}
          (b^*Ab)(d^*Ad).
\]
The constraints define a nonempty compact set, so the minimum is attained.

\begin{lemma}\label{lem:Rayleigh}
For every $t\in[0,1]$,
\begin{empheq}[left={D(t)=\empheqlbrace}]{alignat=2}
 &\lambda_1\lambda_2(1-t^2),\qquad
       &&\textup{if }0\le t\le\frac{\lambda_2-\lambda_1}{\lambda_2+\lambda_1},\label{eq:D-first}\\[5pt]
 &\frac{1}{4}\bigl[\lambda_1(1+t)+\lambda_2(1-t)\bigr]^2,\qquad
       &&\textup{if }\frac{\lambda_2-\lambda_1}{\lambda_2+\lambda_1}<t\le1.\label{eq:D-second}
\end{empheq}
Each minimum is attained by vectors in a plane spanned by orthonormal
eigenvectors of $A$ corresponding to the eigenvalues $\lambda_1$ and $\lambda_2$.

If $\lambda_1<\lambda_2$ and ${(\lambda_2-\lambda_1)/(\lambda_2+\lambda_1)}<t<1$, unit vectors $b,d\in\C^n$
with $b^*d=t$ attain $D(t)$ if and only if they have the
form~\eqref{eq:right-vectors} for orthonormal eigenvectors $e_1,e_2$
satisfying $Ae_1=\lambda_1e_1$ and $Ae_2=\lambda_2e_2$.
\end{lemma}

\begin{proof}
If $\lambda_1=\lambda_2$, then the eigenspace associated
with $\lambda_1$ has dimension at least two. Every Rayleigh quotient is at least
$\lambda_1$, and choosing $b,d$ in this eigenspace with
$\abs{b^*d}=t$ gives $D(t)=\lambda_1^2$ for every $t\in[0,1]$.
If $t=1$, equality in the Cauchy--Schwarz inequality gives
$d=(b^*d)b$ with $\abs{b^*d}=1$. Hence
\[
 (b^*Ab)(d^*Ad)=\abs{b^*d}^2(b^*Ab)^2=(b^*Ab)^2\ge\lambda_1^2.
\]
Taking $b=d$ to be a unit eigenvector of $A$ corresponding to the eigenvalue
$\lambda_1$ attains this value. These observations prove the formulas whenever
$\lambda_1=\lambda_2$ or $t=1$.

Assume now that $\lambda_1<\lambda_2$ and $0\le t<1$.
Multiplying $d$ by a scalar of modulus one preserves $\abs{b^*d}$
and both Rayleigh quotients, so we may assume $b^*d=t$.
Since $t<1$, define $e_1,e_2$ as in~\eqref{eq:right-basis}.
The construction in the proof of~\cref{lem:numerator} shows that they
form an orthonormal basis of $\spanof\{b,d\}$ and satisfy~\eqref{eq:right-vectors}.

Define an operator $B$ on this plane by
\[
 Bv=(e_1^*Av)e_1+(e_2^*Av)e_2,\qquad v\in\spanof\{b,d\}.
\]
Thus $Bv$ is the orthogonal projection of $Av$ onto the plane $\spanof\{b,d\}$.

\newpage
Since $e_1,e_2$ form an orthonormal basis of
$\spanof\{b,d\}=\spanof\{e_1,e_2\}$, every $v\in\spanof\{b,d\}$ satisfies
$v=(e_1^*v)e_1+(e_2^*v)e_2$. Therefore,
\[
 v^*Bv=(v^*e_1)(e_1^*Av)+(v^*e_2)(e_2^*Av)
 =\bigl[(e_1^*v)e_1+(e_2^*v)e_2\bigr]^*Av=v^*Av.
\]
Since $A$ is positive definite, $v^*Bv=v^*Av>0$ for every nonzero
$v\in\spanof\{b,d\}$, so $B$ is positive definite.
Because $A=A^*$, we also have
$e_i^*Be_j=e_i^*Ae_j=\overline{e_j^*Ae_i}=\overline{e_j^*Be_i}$
for $i,j\in\{1,2\}$. Thus $B$ is also Hermitian.

Let $\mu_1\le\mu_2$ be the eigenvalues of $B$. The min--max principle gives
\begin{align*}
 \mu_1
 &=\min_{\substack{v\in\spanof\{b,d\}\\\norm{v}=1}}v^*Av
 \ge\min_{\substack{v\in\C^n\\\norm{v}=1}}v^*Av=\lambda_1,\\
 \mu_2
 &=\max_{\substack{v\in\spanof\{b,d\}\\\norm{v}=1}}v^*Av
 \ge\min_{\substack{V\subseteq\C^n\\\dim V=2}}\;
       \max_{\substack{v\in V\\\norm{v}=1}}v^*Av=\lambda_2.
\end{align*}
Choose orthonormal eigenvectors $q_1,q_2$ of $B$ in the plane, with
$Bq_i=\mu_iq_i$ for $i\in\{1,2\}$. Then, for every $v\in\spanof\{b,d\}$,
\[
 Bv=\mu_1(q_1^*v)q_1+\mu_2(q_2^*v)q_2.
\]
We define an operator $H$ on $\spanof\{b,d\}$ such that for every
$v\in\spanof\{b,d\}$,
\[
 Hv=\lambda_1(q_1^*v)q_1+\lambda_2(q_2^*v)q_2.
\]
The operator $H$ is positive definite Hermitian with eigenvalues
$\lambda_1,\lambda_2$. Moreover,
\[
 v^*(B-H)v=(\mu_1-\lambda_1)\abs{q_1^*v}^2
             +(\mu_2-\lambda_2)\abs{q_2^*v}^2\ge0.
\]
Thus $b^*Bb\ge b^*Hb>0$ and $d^*Bd\ge d^*Hd>0$, which give
\begin{equation}\label{eq:Rayleigh-comparison}
 (b^*Ab)(d^*Ad)=(b^*Bb)(d^*Bd)\ge(b^*Hb)(d^*Hd).
\end{equation}
Since $e_1,e_2$ form another orthonormal basis of the plane, the matrix
of $H$ in this basis has entries $h_{ij}=e_i^*He_j$ for $i,j\in\{1,2\}$.
As $H$ is Hermitian with eigenvalues $\lambda_1,\lambda_2$, its matrix
takes the form
\[
 \begin{pmatrix}
 h_{11}&h_{12}\\
 \overline{h_{12}}&\lambda_1+\lambda_2-h_{11}
 \end{pmatrix},
 \qquad \lambda_1\le h_{11}\le\lambda_2.
\]
Here the lower-right entry follows from $\tr H=\lambda_1+\lambda_2$,
while the bounds on $h_{11}$ follow from the Rayleigh quotient
characterization. Since $\det H=\lambda_1\lambda_2$, we also have
\[
 \abs{h_{12}}^2=h_{11}(\lambda_1+\lambda_2-h_{11})-\lambda_1\lambda_2
 =(h_{11}-\lambda_1)(\lambda_2-h_{11}).
\]
Using the expressions for $b,d$ in~\eqref{eq:right-vectors}, we obtain
\begin{align*}
 b^*Hb
 &=\frac{1}{2}(1+t)h_{11}
   +\frac{1}{2}(1-t)(\lambda_1+\lambda_2-h_{11})
   +\sqrt{1-t^2}\operatorname{Re}h_{12}\\
 &=th_{11}+\frac{1}{2}(\lambda_1+\lambda_2)(1-t)
   +\sqrt{1-t^2}\operatorname{Re}h_{12},\\
 d^*Hd
 &=th_{11}+\frac{1}{2}(\lambda_1+\lambda_2)(1-t)
   -\sqrt{1-t^2}\operatorname{Re}h_{12}.
\end{align*}
Multiplying the two Rayleigh quotients and completing the square gives
\begin{align}
 (b^*Hb)(d^*Hd)
 &=\left[th_{11}+\frac{1}{2}(\lambda_1+\lambda_2)(1-t)\right]^2
       -(1-t^2)(\operatorname{Re}h_{12})^2\notag\\
 &\ge\left[th_{11}+\frac{1}{2}(\lambda_1+\lambda_2)(1-t)\right]^2
       -(1-t^2)(h_{11}-\lambda_1)(\lambda_2-h_{11})\notag\\
 &=h_{11}^2-(\lambda_1+\lambda_2)(1-t)h_{11}
   +\frac{1}{4}(\lambda_1+\lambda_2)^2(1-t)^2+\lambda_1\lambda_2(1-t^2)\notag\\
 &=\left[h_{11}-\frac{1}{2}(\lambda_1+\lambda_2)(1-t)\right]^2
       +\lambda_1\lambda_2(1-t^2).
 \label{eq:Rayleigh-square}
\end{align}\nopagebreak[4]
Because $0\le t<1$, equality in~\eqref{eq:Rayleigh-square} holds exactly when $h_{12}$ is real.

\begingroup
\emergencystretch=1em
The square in~\eqref{eq:Rayleigh-square} vanishes at
$h_{11}=\frac{1}{2}(\lambda_1+\lambda_2)(1-t)$.
For $0\le t\le{(\lambda_2-\lambda_1)/(\lambda_2+\lambda_1)}$, the minimum of the right-hand side of~\eqref{eq:Rayleigh-square}
over $h_{11}\in[\lambda_1,\lambda_2]$ is $\lambda_1\lambda_2(1-t^2)$.
For ${(\lambda_2-\lambda_1)/(\lambda_2+\lambda_1)}<t<1$, the value at which the square vanishes is less than
$\lambda_1$. Thus its minimum on $[\lambda_1,\lambda_2]$ occurs at
$h_{11}=\lambda_1$, and~\eqref{eq:Rayleigh-square} gives
\begin{align}
 (b^*Hb)(d^*Hd)
 &\ge\left[\lambda_1-\frac{1}{2}(\lambda_1+\lambda_2)(1-t)\right]^2
    +\lambda_1\lambda_2(1-t^2)\notag\\
 &=\frac{1}{4}\bigl[\lambda_1(1+t)+\lambda_2(1-t)\bigr]^2.
 \label{eq:Rayleigh-minimum}
\end{align}
\endgroup

To attain these lower bounds for $A$, let
\[
 m=\max\left\{\lambda_1,\frac{1}{2}(\lambda_1+\lambda_2)(1-t)\right\}.
\]
Since $\frac{1}{2}(\lambda_1+\lambda_2)(1-t)\le\lambda_2$, the value
$m\in[\lambda_1,\lambda_2]$ minimizes the square in~\eqref{eq:Rayleigh-square}.
Let $H_0$ be the Hermitian
operator on $\spanof\{b,d\}$ whose matrix in the basis $e_1,e_2$ is
\[
 \begin{pmatrix}
  m&\sqrt{(m-\lambda_1)(\lambda_2-m)}\\
  \sqrt{(m-\lambda_1)(\lambda_2-m)}&\lambda_1+\lambda_2-m
 \end{pmatrix}.
\]
This matrix has trace $\lambda_1+\lambda_2$ and determinant
$\lambda_1\lambda_2$, so $H_0$ has eigenvalues $\lambda_1,\lambda_2$.
Because the off-diagonal entries are real and the upper-left entry
minimizes the square in~\eqref{eq:Rayleigh-square},
$(b^*H_0b)(d^*H_0d)$ attains the lower bound.
Choose an orthonormal eigenbasis $r_1,r_2$ of $H_0$ on
$\spanof\{b,d\}$ and orthonormal eigenvectors $p_1,p_2$ of $A$
such that $H_0r_i=\lambda_i r_i$ and $Ap_i=\lambda_i p_i$ for $i\in\{1,2\}$.
Define a linear map $U:\spanof\{b,d\}\to\C^n$ such that for every
$v\in\spanof\{b,d\}$,
\[
 Uv=(r_1^*v)p_1+(r_2^*v)p_2.
\]
Thus $Ur_1=p_1$ and $Ur_2=p_2$.
Furthermore, for all $v,w\in\spanof\{b,d\}$,
\begin{align*}
 (Uv)^*(Uw)&=(v^*r_1)(r_1^*w)+(v^*r_2)(r_2^*w)=v^*w,\\
 (Uv)^*A(Uv)&=\lambda_1\abs{r_1^*v}^2+\lambda_2\abs{r_2^*v}^2=v^*H_0v.
\end{align*}
Consequently, $\norm{Ub}=\norm{Ud}=1$, $(Ub)^*(Ud)=b^*d=t$, and
\[
 \bigl[(Ub)^*A(Ub)\bigr]\bigl[(Ud)^*A(Ud)\bigr]=(b^*H_0b)(d^*H_0d).
\]
The vectors $Ub,Ud\in\spanof\{p_1,p_2\}$ therefore attain the minimum
for $A$. Together with the cases treated at the start, this proves~\eqref{eq:D-first} and~\eqref{eq:D-second}.

Now suppose $\lambda_1<\lambda_2$ and ${(\lambda_2-\lambda_1)/(\lambda_2+\lambda_1)}<t<1$, and let
unit vectors $b,d$ satisfy $b^*d=t$ and attain $D(t)$, so
$(b^*Ab)(d^*Ad)=D(t)$.
Then
\[
 D(t)=(b^*Ab)(d^*Ad)=(b^*Bb)(d^*Bd)
 \ge(b^*Hb)(d^*Hd)
 \ge\frac{1}{4}\bigl[\lambda_1(1+t)+\lambda_2(1-t)\bigr]^2=D(t),
\]
\noindent where the second equality and the first inequality follow
from~\eqref{eq:Rayleigh-comparison}, the last inequality uses the
minimum value in~\eqref{eq:Rayleigh-minimum}, and the last equality
follows from~\eqref{eq:D-second}.
Thus both inequalities are equalities.
The minimum of the square in~\eqref{eq:Rayleigh-square} is unique on
$[\lambda_1,\lambda_2]$, so $h_{11}=\lambda_1$ and $h_{12}=0$.
Consequently $He_1=\lambda_1e_1$ and $He_2=\lambda_2e_2$.
Moreover, $b^*Bb\ge b^*Hb>0$ and $d^*Bd\ge d^*Hd>0$.
Equality of the products therefore requires equality of each factor.
Since $B-H$ is positive semidefinite,
\[
 \norm{(B-H)^{\frac{1}{2}}b}^2=b^*(B-H)b=0,
 \qquad
 \norm{(B-H)^{\frac{1}{2}}d}^2=d^*(B-H)d=0.
\]
Hence $(B-H)b=(B-H)d=0$.
The vectors $b,d$ span the plane because $t<1$, so $B=H$ and
\[
 e_1^*Ae_1=e_1^*Be_1=\lambda_1,
 \qquad e_2^*Ae_2=e_2^*Be_2=\lambda_2.
\]
The first equality attains the minimum Rayleigh quotient of $A$, so
$Ae_1=\lambda_1e_1$.
Since $e_2\perp e_1$ and $e_2^*Ae_2=\lambda_2$, the equality case of
the Courant--Fischer theorem gives $Ae_2=\lambda_2e_2$.

Conversely, these eigenvector relations and~\eqref{eq:right-vectors} give
\[
 b^*Ab=d^*Ad=\frac{1}{2}\bigl[\lambda_1(1+t)+\lambda_2(1-t)\bigr],
 \qquad
 (b^*Ab)(d^*Ad)=D(t).
\]
This proves the equality characterization.
\end{proof}

Lemma~\ref{lem:Rayleigh} chooses the vectors $b,d$ to minimize
the denominator at a fixed $t$. Lemma~\ref{lem:numerator} applies to
$b,d$ and chooses $a,c$ to attain the numerator bound.
The next proposition combines these two choices.

\begin{proposition}\label{prop:reduction}
The sharp constant in~\eqref{eq:C-def} is
\begin{equation}\label{eq:scalar-reduction}
 \CalC_A=\max_{0\le t\le1}\frac{F(t)}{D(t)}.
\end{equation}
The maximum of $F(t)/D(t)$ over $t\in[0,1]$ is attained
by rank-one matrices whose row and column spaces
lie in a plane spanned by
eigenvectors of $A$ corresponding to the eigenvalues $\lambda_1$ and $\lambda_2$.
\end{proposition}

\begin{proof}
The quotient in~\eqref{eq:C-def} is unchanged when either matrix is
multiplied by a nonzero scalar. We may therefore normalize
$\norm{X}_{\F}=\norm{Y}_{\F}=1$ and write $X=ab^*$, $Y=cd^*$ with
unit vectors $a,b,c,d\in\C^n$. The product of the four unit spheres is
compact, and the denominator is at least $\lambda_1^2>0$ there.
Thus the quotient is continuous on a compact set and attains its maximum.
For $t=\abs{b^*d}$,
\[
 \frac{\norm{XY-YX}_{\F}^{2}}{\norm{X}_A^2\norm{Y}_A^2}
 =\frac{\norm{ab^*cd^*-cd^*ab^*}_{\F}^{2}}{(b^*Ab)(d^*Ad)}
 \le\frac{F(t)}{(b^*Ab)(d^*Ad)}\le\frac{F(t)}{D(t)},
\]
\noindent where the inequalities use~\cref{lem:numerator} and the
definitions of $F,D$.
Taking the maximum gives the upper bound in~\eqref{eq:scalar-reduction}.

For the reverse bound, fix $t\in[0,1]$.
By~\cref{lem:Rayleigh}, choose orthonormal eigenvectors $p_1,p_2$ of $A$,
with $Ap_i=\lambda_i p_i$ for $i\in\{1,2\}$, and unit vectors
$b,d\in\spanof\{p_1,p_2\}$ such that
$\abs{b^*d}=t$ and $(b^*Ab)(d^*Ad)=D(t)$.
Multiplying $d$ by a scalar of modulus one if necessary, we may assume
$b^*d=t$.
Choose $x\in[0,1]$ attaining $F(t)$.
If $0\le t<1$, the linearly independent vectors $b,d$ span the
two-dimensional space $\spanof\{p_1,p_2\}$.
The orthonormal vectors $e_1,e_2$ defined by~\eqref{eq:right-basis}
therefore satisfy
\[
 \spanof\{e_1,e_2\}=\spanof\{b,d\}=\spanof\{p_1,p_2\}.
\]
If $t=1$, then $d=b$. In this case, let $e_1=b$ and choose a unit vector
$e_2\in\spanof\{p_1,p_2\}$ with $e_2\perp e_1$.
In either case, let $u_1=e_1$ and $u_2=e_2$, and define $a,c$
by~\eqref{eq:left-vectors}.
By the construction in the proof of~\cref{lem:numerator},
$a,c$ are unit vectors in $\spanof\{p_1,p_2\}$ and satisfy
\[
 \norm{ab^*cd^*-cd^*ab^*}_{\F}^{2}=f(x,t)=F(t).
\]
Consequently,
\[
 \frac{\norm{ab^*cd^*-cd^*ab^*}_{\F}^{2}}{(b^*Ab)(d^*Ad)}
 =\frac{F(t)}{D(t)}.
\]
Thus every value $F(t)/D(t)$ is attained by a rank-one pair.
Taking the maximum over $t\in[0,1]$ proves~\eqref{eq:scalar-reduction}.
\end{proof}

\section{Determination of the sharp constant}\label{sec:scalar}

By~\cref{prop:reduction}, it remains to maximize $F(t)/D(t)$ on $[0,1]$.
We first compute $F,F'$ at $x_*(t)$ and identify a strictly increasing
function $h$ that determines the sign of the derivative of $F(t)/D(t)$.

\Needspace{13\baselineskip}
\begin{lemma}\label{lem:scalar}
For every $t\in(0,1)$, let $x=x_*(t)$ be the unique maximizer
from~\cref{lem:x-optimum}. Then
\begin{equation}\label{eq:F-value}
 F(t)=\frac{2(1+xt)^3}{1+t^2+2xt},\qquad
 F'(t)=-\frac{(t^2-x^2)(1+xt)}{x(1+t^2+2xt)}<0.
\end{equation}
Moreover, the function
\begin{equation}\label{eq:h-def}
 h(t)=\frac{t^2-x_*(t)^2}{\bigl[t+3x_*(t)\bigr]\bigl[1+t x_*(t)\bigr]}
     =-\frac{F'(t)}{2F(t)-tF'(t)}
\end{equation}
is strictly increasing on $(0,1)$, with
$\lim_{t\downarrow0}h(t)=0$ and $\lim_{t\uparrow1}h(t)=1$.
It satisfies $0<h(t)<t$ for all $t\in(0,1)$.
\end{lemma}

\begin{proof}
Throughout this proof, $x$ denotes $x_*(t)$ and therefore varies with $t$.
Under this abuse of notation,~\eqref{eq:F-def} becomes $F(t)=f(x,t)$.
By~\cref{lem:x-optimum}, $0<x<t<1$ and $f_x(x,t)=0$.
Using $f_x(x,t)=0$ in~\eqref{eq:fx}, we obtain
\[
 2xt^2\sqrt{1-x^2}=\sqrt{1-t^2}(t-x-2tx^2).
\]
Multiplying both sides by $\sqrt{1-t^2}/(2xt^2)$ and then
using~\eqref{eq:x-cubic}, we obtain
\begin{align}
 \sqrt{(1-t^2)(1-x^2)}
 &=\frac{(1-t^2)(t-x-2tx^2)}{2xt^2}\notag\\
 &=\frac{(t-x)(1-t^2)-2tx^2(1-t^2)}{2xt^2}\notag\\
 &=\frac{4tx^2(1+xt)-2tx^2(1-t^2)}{2xt^2}\notag\\
 &=\frac{x(1+t^2+2xt)}{t}.
 \label{eq:root-identity}
\end{align}
Substituting~\eqref{eq:root-identity} into the definition of $f$ gives
\[
 F(t)=(1+xt)\left[1-xt+\sqrt{(1-x^2)(1-t^2)}\right]
 =\frac{(1+xt)(t+x+2tx^2)}{t}
 =\frac{2(1+xt)^3}{1+t^2+2xt},
\]
\noindent where the last equality follows from~\eqref{eq:x-cubic} through
\[
 (t+x+2tx^2)(1+t^2+2xt)-2t(1+xt)^2
 =4tx^2(1+xt)-(t-x)(1-t^2)=0.
\]

Since $F(t)=f(x_*(t),t)$ and $f_x(x_*(t),t)=0$, differentiation gives
\begin{align*}
 F'(t)
 &=f_x(x,t)x_*'(t)+f_t(x,t)\\
 &=f_t(x,t)\\
 &=-2x^2t+\frac{(1-x^2)(x-t-2xt^2)}{\sqrt{(1-x^2)(1-t^2)}}\\
 &=-2x^2t+\frac{t(1-x^2)(x-t-2xt^2)}{x(1+t^2+2xt)}\\
 &=\frac{-2x^3t(1+t^2+2xt)+t(1-x^2)(x-t-2xt^2)}{x(1+t^2+2xt)}\\
 &=-\frac{(t^2-x^2)(1+xt)}{x(1+t^2+2xt)},
\end{align*}
\noindent where the fourth equality uses~\eqref{eq:root-identity}
and the last equality follows from~\eqref{eq:x-cubic}.
All factors in the final fraction are positive,
so $F'(t)<0$. Since $F(t)>0$, we also have $2F(t)-tF'(t)>0$.
Consequently,
\[
 -\frac{F'(t)}{2F(t)-tF'(t)}
 =\frac{t^2-x^2}{4x(1+xt)^2+t(t^2-x^2)}
 =\frac{t^2-x^2}{(t+3x)(1+xt)}=h(t),
\]
\noindent where the second equality follows from~\eqref{eq:x-cubic} through
\[
 4x(1+xt)^2+t(t^2-x^2)-(t+3x)(1+xt)
 =4tx^2(1+xt)-(t-x)(1-t^2)=0.
\]

To prove that $h$ is strictly increasing, we factor it as
\[
 h(t)=\frac{t^2-x^2}{t(t+3x)}\,\frac{t}{1+xt}.
\]
Both factors are positive because $0<x<t<1$.
Differentiating~\eqref{eq:x-cubic} implicitly with respect to $t$,
treating $x=x_*(t)$ as a function of $t$, yields
\[
 x_*'(t)=\frac{1-3t^2+2tx-4x^2-8tx^3}{12t^2x^2+8tx+1-t^2}.
\]
Thus
\[
 \frac{\mathrm d}{\mathrm dt}\left(\frac{x_*(t)}{t}\right)
 =\frac{t x_*'(t)-x}{t^2}
 =-\frac{2\bigl[8tx^3+4x^2+t(t-x)\bigr]}
          {t(12t^2x^2+8tx+1-t^2)}<0,
\]
\noindent where the second equality uses~\eqref{eq:x-cubic}
and the inequality follows from $0<x<t<1$.
So $t x_*'(t)<x$ and $t^2x_*'(t)<tx<1$.
Differentiating the two factors of $h$ therefore gives
\[
 \frac{\mathrm d}{\mathrm dt}\left(\frac{t^2-x^2}{t(t+3x)}\right)
 =-\frac{3t^2+2tx+3x^2}{(t+3x)^2}
       \frac{\mathrm d}{\mathrm dt}\left(\frac{x_*(t)}{t}\right)>0,
 \qquad
 \frac{\mathrm d}{\mathrm dt}\left(\frac{t}{1+xt}\right)
 =\frac{1-t^2x_*'(t)}{(1+xt)^2}>0.
\]
Their product is strictly increasing. Also,
\[
 0<\frac{t^2-x^2}{t(t+3x)}<1,\qquad
 0<\frac{t}{1+xt}<t.
\]
Multiplying these positive inequalities gives $0<h(t)<t$.

Since $0<h(t)<t$, we have $h(t)\to0$ as $t\downarrow0$.
For the other endpoint,~\eqref{eq:x-cubic} gives
\[
 0<x_*(t)^2
 =\frac{\bigl[t-x_*(t)\bigr](1-t^2)}{4t\bigl[1+t x_*(t)\bigr]}
 \le\frac{1-t^2}{4}\longrightarrow0
 \qquad (t\uparrow1).
\]
Substitution into~\eqref{eq:h-def} now gives $h(t)\to1$.
\end{proof}

\Needspace{10\baselineskip}
\begin{proof}[Proof of \cref{thm:main}]
If $\lambda_1=\lambda_2$, then $D(t)=\lambda_1^2$ for every $t$.
Since $F$ is strictly decreasing in $[0,1]$,
\[
 \CalC_A=\max_{0\le t\le1}\frac{F(t)}{\lambda_1^2}
 =\frac{F(0)}{\lambda_1^2}=\frac{2}{\lambda_1^2}.
\]
Let $e_1,e_2\in\C^n$ be orthonormal eigenvectors of $A$
corresponding to $\lambda_1$, and let $X=e_1e_2^*$ and $Y=e_2e_1^*$.
For real symmetric $A$, choose real orthonormal eigenvectors $e_1,e_2$
corresponding to $\lambda_1$, and let $X=e_1e_2^\top$ and $Y=e_2e_1^\top$.
In both cases,
\[
 \norm{XY-YX}_{\F}^2=2,\qquad
 \norm{X}_A^2=\norm{Y}_A^2=\lambda_1.
\]
Thus $X,Y$ attain the bound $2/\lambda_1^2$.

Assume now that $\lambda_1<\lambda_2$, so
$0<{(\lambda_2-\lambda_1)/(\lambda_2+\lambda_1)}<1$.
By~\cref{lem:scalar}, $h$ is continuous and strictly increasing in $(0,1)$,
with $\lim_{t\downarrow0}h(t)=0$ and $\lim_{t\uparrow1}h(t)=1$.
Thus there is exactly one $t_0\in(0,1)$ satisfying
\begin{equation}\label{eq:optimal-h}
 h(t_0)=\frac{\lambda_2-\lambda_1}{\lambda_2+\lambda_1}.
\end{equation}
Since $0<h(t)<t$ for all $t\in(0,1)$ by~\cref{lem:scalar},
we have ${(\lambda_2-\lambda_1)/(\lambda_2+\lambda_1)}<t_0<1$.
We now check the derivative of $F(t)/D(t)$.
By~\eqref{eq:h-def}, we have $F'(t)=-h(t)\bigl[2F(t)-tF'(t)\bigr]$.
Consequently, for any real number $s$,
\begin{align}
 (1-st)F'(t)+2sF(t)
 &=F'(t)+s\bigl[2F(t)-tF'(t)\bigr]
 =\bigl[s-h(t)\bigr]\bigl[2F(t)-tF'(t)\bigr].
 \label{eq:derivative-factorization}
\end{align}
For $0<t<{(\lambda_2-\lambda_1)/(\lambda_2+\lambda_1)}$, applying~\eqref{eq:D-first} and then
\eqref{eq:derivative-factorization} with $s=t$ gives
\[
 \frac{\mathrm d}{\mathrm dt}\left(\frac{F(t)}{D(t)}\right)
 =\frac{(1-t^2)F'(t)+2tF(t)}{\lambda_1\lambda_2(1-t^2)^2}
 =\frac{\bigl[2F(t)-tF'(t)\bigr]\bigl[t-h(t)\bigr]}
        {\lambda_1\lambda_2(1-t^2)^2}>0.
\]
For ${(\lambda_2-\lambda_1)/(\lambda_2+\lambda_1)}<t<1$, applying~\eqref{eq:D-second} and then
\eqref{eq:derivative-factorization} with $s={(\lambda_2-\lambda_1)/(\lambda_2+\lambda_1)}$ gives
\begin{align*}
 \frac{\mathrm d}{\mathrm dt}\left(\frac{F(t)}{D(t)}\right)
 &=\frac{4\Bigl(\bigl[\lambda_1(1+t)+\lambda_2(1-t)\bigr]F'(t)
              +2(\lambda_2-\lambda_1)F(t)\Bigr)}
        {\bigl[\lambda_1(1+t)+\lambda_2(1-t)\bigr]^3}\\
 &=\frac{4\bigl[2F(t)-tF'(t)\bigr]
              \bigl[(\lambda_2-\lambda_1)-(\lambda_1+\lambda_2)h(t)\bigr]}
        {\bigl[\lambda_1(1+t)+\lambda_2(1-t)\bigr]^3}.
\end{align*}
By~\eqref{eq:optimal-h}, the sign of the last derivative equals the sign of $h(t_0)-h(t)$,
since all remaining factors are positive.
Thus it is positive for ${(\lambda_2-\lambda_1)/(\lambda_2+\lambda_1)}<t<t_0$
and negative for $t_0<t<1$.
Together with the continuity of $F(t)/D(t)$, these derivative signs show that
$F(t)/D(t)$ increases strictly on $[0,t_0]$ and decreases strictly on $[t_0,1]$. Hence $t_0$ is the unique global maximizer.
Let $x_0=x_*(t_0)$. Substituting $t=t_0$ into~\eqref{eq:x-cubic}
gives the first equation in~\eqref{eq:optimal-system}, while
\eqref{eq:h-def} and~\eqref{eq:optimal-h} give the second.

Conversely, let $(x,t)$ be any solution of~\eqref{eq:optimal-system}
with $0<x<t<1$.
By~\cref{lem:x-optimum}, the first equation gives $x=x_*(t)$.
The second equation, together with~\eqref{eq:h-def} and~\eqref{eq:optimal-h}, then gives
$h(t)=h(t_0)$.
Since $h$ is strictly increasing, $t=t_0$ and $x=x_*(t_0)=x_0$.
Because ${(\lambda_2-\lambda_1)/(\lambda_2+\lambda_1)}<t_0<1$, substituting~\eqref{eq:D-second}
and~\eqref{eq:F-value} into~\eqref{eq:scalar-reduction} yields
\[
 \CalC_A=\frac{F(t_0)}{D(t_0)}
 =\frac{4F(t_0)}{\bigl[\lambda_1(1+t_0)+\lambda_2(1-t_0)\bigr]^2}
 =\frac{8(1+x_0t_0)^3}
 {\bigl[\lambda_1(1+t_0)+\lambda_2(1-t_0)\bigr]^2(1+t_0^2+2x_0t_0)}.
\]
For real symmetric $A$, choose real eigenvectors of $A$ corresponding to
the eigenvalues $\lambda_1,\lambda_2$.
The attaining vectors in~\cref{lem:Rayleigh,lem:numerator} have real
coefficients in that basis, so real matrices attain the same sharp constant
as in the complex case.
This proves the real assertion as well.
\end{proof}

\section{Equality cases and consequences}\label{sec:equality}

We now determine all rank-one pairs $X,Y\in M_n(\C)$ satisfying
\begin{equation}\label{eq:sharp-equality}
 \norm{XY-YX}_{\F}^{2}=\CalC_A\norm{X}_A^2\norm{Y}_A^2.
\end{equation}

\begin{theorem}\label{thm:equality}
Let $X,Y\in M_n(\C)$ have rank one.

\noindent\textup{(1)} If $\lambda_1=\lambda_2$, then~\eqref{eq:sharp-equality} holds if and only if
\begin{equation}\label{eq:equality-equal}
 X=\xi e_1e_2^*,\qquad Y=\eta e_2e_1^*,
\end{equation}
where $\xi,\eta\in\C\setminus\{0\}$ and $e_1,e_2\in\C^n$ are
orthonormal eigenvectors of $A$ corresponding to the eigenvalue $\lambda_1$.

\addvspace{\baselineskip}
\noindent\textup{(2)} If $\lambda_1<\lambda_2$, let $(x_0,t_0)$ be the unique pair in
\cref{thm:main}. Then~\eqref{eq:sharp-equality} holds if and only if
$X=\xi ab^*$ and $Y=\eta cd^*$ for $\xi,\eta\in\C\setminus\{0\}$
and unit vectors $a,b,c,d\in\C^n$ of the form
\begin{align}
 a&=\sqrt{\frac{1-x_0}{2}}\,e_1-\sqrt{\frac{1+x_0}{2}}\,e_2,
 &c&=\sqrt{\frac{1-x_0}{2}}\,e_1+\sqrt{\frac{1+x_0}{2}}\,e_2,
 \label{eq:extremal-left}\\
 b&=\sqrt{\frac{1+t_0}{2}}\,e_1+\sqrt{\frac{1-t_0}{2}}\,e_2,
 &d&=\sqrt{\frac{1+t_0}{2}}\,e_1-\sqrt{\frac{1-t_0}{2}}\,e_2,
 \label{eq:extremal-right}
\end{align}
\noindent where $e_1,e_2\in\C^n$ are orthonormal eigenvectors of $A$
corresponding to the eigenvalues $\lambda_1$ and $\lambda_2$, respectively.
\end{theorem}

\Needspace{6\baselineskip}
\begin{proof}
Write $X=\alpha \widetilde a b^*$ and $Y=\beta \widetilde c d^*$, where $\widetilde a,b,\widetilde c,d\in\C^n$
are unit vectors and $\alpha,\beta\in\C\setminus\{0\}$.
Let $x=\abs{\widetilde a^*\widetilde c}$ and $t=\abs{b^*d}$.
Multiplying $\widetilde c,d$ by suitable scalars of modulus one and adjusting
$\beta$ accordingly, we may assume $\widetilde a^*\widetilde c=-x$ and $b^*d=t$
without changing $X,Y$.
The scalar factors cancel, and
\[
 \frac{\norm{XY-YX}_{\F}^{2}}{\norm{X}_A^2\norm{Y}_A^2}
 =\frac{\norm{\widetilde a b^*\widetilde c d^*-\widetilde c d^*\widetilde a b^*}_{\F}^{2}}{(b^*Ab)(d^*Ad)}
 \le\frac{f(x,t)}{(b^*Ab)(d^*Ad)}
 \le\frac{F(t)}{(b^*Ab)(d^*Ad)}
 \le\frac{F(t)}{D(t)}\le\CalC_A.
\]
Therefore~\eqref{eq:sharp-equality}
holds if and only if equality holds at each of these four steps.

\addvspace{\baselineskip}
\noindent\textup{(1)} Suppose that $\lambda_1=\lambda_2$ and equality holds.
Since $D(t)=\lambda_1^2$ and $F$ is strictly decreasing, equality in the
last inequality forces $t=0$.
Then $f(x,0)=1+\sqrt{1-x^2}$ equals $F(0)=2$ only at $x=0$.
Hence $b\perp d$ and $\widetilde a\perp\widetilde c$.
Equality in~\cref{lem:numerator} gives
\[
 \norm{\widetilde a b^*\widetilde c d^*-\widetilde c d^*\widetilde a b^*}_{\F}^{2}=f(0,0)=2.
\]
Moreover, since $\tr(d\widetilde a^*\widetilde c b^*)=(\widetilde a^*\widetilde c)(b^*d)=0$ and
$\norm{\widetilde a d^*}_{\F}=\norm{\widetilde c b^*}_{\F}=1$, expanding the squared Frobenius norm yields
\[
 2=\norm{\widetilde a b^*\widetilde c d^*-\widetilde c d^*\widetilde a b^*}_{\F}^{2}
  =\abs{b^*\widetilde c}^2+\abs{d^*\widetilde a}^2\le2.
\]
The equality case of the Cauchy--Schwarz inequality gives
$\widetilde c=(b^*\widetilde c)b$ and $\widetilde a=(d^*\widetilde a)d$, with
$\abs{b^*\widetilde c}=\abs{d^*\widetilde a}=1$.
Equality in the denominator also gives
$(b^*Ab)(d^*Ad)=\lambda_1^2$.
As each factor is at least $\lambda_1$, both must equal $\lambda_1$.
Thus $b,d$ are orthonormal eigenvectors of $A$ corresponding to the
eigenvalue $\lambda_1$.
Let $e_1=d$, $e_2=b$, $\xi=\alpha(d^*\widetilde a)$, and $\eta=\beta(b^*\widetilde c)$.
Then $X=\xi e_1e_2^*$ and $Y=\eta e_2e_1^*$, with
$\xi,\eta\in\C\setminus\{0\}$, proving~\eqref{eq:equality-equal}.

Conversely, the matrices in~\eqref{eq:equality-equal} satisfy
$XY-YX=\xi\eta(e_1e_1^*-e_2e_2^*)$ and
$\norm{X}_A^2\norm{Y}_A^2=\lambda_1^2\abs{\xi\eta}^2$. Hence
\[
 \frac{\norm{XY-YX}_{\F}^{2}}{\norm{X}_A^2\norm{Y}_A^2}
 =\frac{\abs{\xi\eta}^2\norm{e_1e_1^*-e_2e_2^*}_{\F}^{2}}
         {\lambda_1^2\abs{\xi\eta}^2}
 =\frac{2}{\lambda_1^2}=\CalC_A,
\]
\noindent where $e_1\perp e_2$ makes the two rank-one projections
orthogonal in the Frobenius inner product.

\addvspace{\baselineskip}
\noindent\textup{(2)} Suppose that $\lambda_1<\lambda_2$ and equality holds.
The unique maximum of $F(t)/D(t)$ forces $t=t_0$.
The unique maximum of $f(x,t_0)$ then forces $x=x_*(t_0)=x_0$.
Since $t_0>{(\lambda_2-\lambda_1)/(\lambda_2+\lambda_1)}$, equality in the denominator and~\cref{lem:Rayleigh}
show that $b,d$ have the forms in~\eqref{eq:extremal-right} for orthonormal eigenvectors $e_1,e_2$
with $Ae_1=\lambda_1e_1$ and $Ae_2=\lambda_2e_2$.
Equality in~\cref{lem:numerator} gives $u_1=ze_1$ and $u_2=ze_2$
for some $z\in\C$ with $\abs{z}=1$.
Let $a,c$ be the vectors in~\eqref{eq:extremal-left}.
Applying~\eqref{eq:left-vectors} to $\widetilde a,\widetilde c$
and substituting $u_1=ze_1$ and $u_2=ze_2$, we obtain
\[
 \widetilde a=za,\qquad \widetilde c=zc.
\]
Let $\xi=\alpha z$ and $\eta=\beta z$.
Then $\xi,\eta\in\C\setminus\{0\}$ and
\[
 X=\alpha\widetilde a b^*=\xi ab^*,\qquad
 Y=\beta\widetilde c d^*=\eta cd^*.
\]
Together with~\eqref{eq:extremal-right}, this proves necessity.

\Needspace{8\baselineskip}
Conversely, the vectors in~\eqref{eq:extremal-left} and
\eqref{eq:extremal-right} satisfy $a^*c=-x_0$, $b^*d=t_0$, and
\[
 b^*Ab=d^*Ad=\frac{1}{2}\bigl[\lambda_1(1+t_0)+\lambda_2(1-t_0)\bigr].
\]
They also satisfy the equality condition in~\cref{lem:numerator}
with $u_1=e_1$, $u_2=e_2$. Hence
\[
 \frac{\norm{XY-YX}_{\F}^{2}}{\norm{X}_A^2\norm{Y}_A^2}
 =\frac{f(x_0,t_0)}{D(t_0)}
 =\frac{F(t_0)}{D(t_0)}=\CalC_A.
\]
This proves sufficiency when $\lambda_1<\lambda_2$.
\end{proof}

The scalar formula~\eqref{eq:scalar-reduction} also gives direct comparisons for the sharp constant.

\begin{corollary}\label{cor:bounds}
For $0<\lambda_1<\lambda_2$,
\begin{equation}\label{eq:strict-lower}
 \max\left\{\frac{2}{\lambda_1\lambda_2},\frac{1}{\lambda_1^2}\right\}
 <\CalC_A
 <\frac{1}{\lambda_1}\left(\frac{1}{\lambda_1}+\frac{1}{\lambda_2}\right).
\end{equation}
\end{corollary}

\begin{proof}
The values at $t=0,1$ are
\[
 \frac{F(0)}{D(0)}=\frac{2}{\lambda_1\lambda_2},
 \qquad \frac{F(1)}{D(1)}=\frac{1}{\lambda_1^2}.
\]
Neither endpoint is the unique maximizer $t_0\in(0,1)$, so both
values are strictly smaller than $\CalC_A$.
The strict upper bound in~\eqref{eq:strict-lower} is the nonoptimality
result of Fang and Cheng~\cite[Section~4]{FC}.
\end{proof}

\begin{corollary}\label{cor:monotonicity}
For fixed $\lambda_1>0$, the sharp constant $\CalC_A$ is strictly decreasing as a
function of $\lambda_2\in[\lambda_1,\infty)$ and tends to $1/\lambda_1^2$
as $\lambda_2\to\infty$.
\end{corollary}

\begin{proof}
Keep $\lambda_1$ fixed and let
$\lambda_1\le\lambda_2<\widetilde\lambda_2$.
Only the two smallest eigenvalues affect the constant, so it is enough
to compare $A=\diag(\lambda_1,\lambda_2)$ with
$\widetilde A=\diag(\lambda_1,\widetilde\lambda_2)$.
Let $t_1$ attain the maximum in~\eqref{eq:scalar-reduction} for $\widetilde A$.
By~\cref{thm:main},
\[
 \frac{\lambda_2-\lambda_1}{\lambda_2+\lambda_1}
 <\frac{\widetilde\lambda_2-\lambda_1}{\widetilde\lambda_2+\lambda_1}
 <t_1<1.
\]
Thus~\eqref{eq:D-second} applies to both $A$ and $\widetilde A$ at $t=t_1$.
Since $\widetilde\lambda_2>\lambda_2$ and $t_1<1$,
\[
 \bigl[\lambda_1(1+t_1)+\widetilde\lambda_2(1-t_1)\bigr]
 -\bigl[\lambda_1(1+t_1)+\lambda_2(1-t_1)\bigr]
 =(\widetilde\lambda_2-\lambda_2)(1-t_1)>0.
\]
Since $F(t_1)>0$,~\eqref{eq:scalar-reduction} and~\eqref{eq:D-second} give
\[
 \CalC_A
 \ge\frac{F(t_1)}{D(t_1)}
 =\frac{4F(t_1)}{\bigl[\lambda_1(1+t_1)+\lambda_2(1-t_1)\bigr]^2}
 >\frac{4F(t_1)}{\bigl[\lambda_1(1+t_1)+\widetilde\lambda_2(1-t_1)\bigr]^2}
 =\CalC_{\widetilde A},
\]
\noindent where the last equality uses the choice of $t_1$.
For fixed $\lambda_1>0$,~\eqref{eq:strict-lower} gives
\[
 0<\CalC_A-\frac{1}{\lambda_1^2}<\frac{1}{\lambda_1\lambda_2}
 \longrightarrow0\qquad(\lambda_2\to\infty).\qedhere
\]
\end{proof}

\begin{corollary}\label{cor:nonnilpotent}
If $\lambda_1<\lambda_2$, then every rank-one pair $X,Y\in M_n(\C)$
satisfying~\eqref{eq:sharp-equality} consists of two matrices that are
neither normal nor nilpotent.
\end{corollary}

\begin{proof}
For unit vectors $a,b$ and every integer $m\ge2$, the identities
\[
 (ab^*)^m=(b^*a)^{m-1}ab^*,\qquad
 (ab^*)(ab^*)^*=aa^*,\qquad (ab^*)^*(ab^*)=bb^*
\]
show that $ab^*$ is nilpotent exactly when $b^*a=0$, and normal
exactly when $aa^*=bb^*$, or equivalently $\abs{b^*a}=1$.

For the unit vectors in~\eqref{eq:extremal-left} and~\eqref{eq:extremal-right},
\[
 0<b^*a=d^*c
 =\frac{1}{2}\left[\sqrt{(1+t_0)(1-x_0)}
                   -\sqrt{(1-t_0)(1+x_0)}\right]
 <\frac{1}{2}\sqrt{(1+t_0)(1-x_0)}<1.
\]
Thus $0<\abs{b^*a},\abs{d^*c}<1$, excluding both nilpotence and
normality for $X$ and $Y$.
\end{proof}

\section{An application to density matrices}\label{sec:density}

A density matrix is a positive semidefinite Hermitian matrix of
trace one. In quantum
mechanics, it represents a state of a finite-dimensional system.

Let $\rho,\sigma\in\C^{n\times n}$ be density matrices.
The matrices $\rho$ and $\sigma$ commute if and only if they are
simultaneously diagonalizable by a unitary matrix. Their diagonal entries then give
probability distributions in a common orthonormal basis.
Iyengar, Chandan, and Srikanth~\cite{ICS} proposed
$2\norm{\rho\sigma-\sigma\rho}_{\F}^2$ as a measure of the
incompatibility of two density matrices and proved
\begin{equation}\label{eq:density-universal}
 \norm{\rho\sigma-\sigma\rho}_{\F}^2\le\frac12.
\end{equation}
The same quantity was used by Jing, Wu, and del Campo~\cite{JWC}
to study the time required for quantum dynamics to generate
incompatibility between an initial state and an evolved state.

For a fixed positive definite Hermitian matrix $A$, the quantities
$\tr(A\rho)$ and $\tr(A\sigma)$ are its expectation values in the
two states. We determine the smallest constant $K_A$ for which
\begin{equation}\label{eq:density-main}
 \norm{\rho\sigma-\sigma\rho}_{\F}^2
 \le K_A\tr(A\rho)\tr(A\sigma)
\end{equation}
holds for all density matrices $\rho,\sigma\in\C^{n\times n}$.
The constant depends only on the two smallest eigenvalues of $A$,
and every equality pair consists of rank-one density matrices.
For every $A$ with $\lambda_1<\lambda_2$,
\cref{cor:density-comparison} also determines the exact maximum of
$\norm{\rho\sigma-\sigma\rho}_{\F}^2$ under any common upper bound
at least $\lambda_1$ on the two expectation values. It identifies the full range of
bounds for which this maximum is strictly smaller than $1/2$.

\begin{theorem}\label{thm:density-main}
The sharp constant in~\eqref{eq:density-main} is
\begin{equation}\label{eq:density-constant}
 K_A=\frac{8\tau^2(1-\tau^2)}
 {\bigl[\lambda_1(1+\tau)+\lambda_2(1-\tau)\bigr]^2},
\end{equation}
where $\tau\in[1/\sqrt2,1)$ is the unique solution of
\begin{equation}\label{eq:density-tau}
 \frac{2\tau^2-1}{\tau^3}
 =\frac{\lambda_2-\lambda_1}{\lambda_2+\lambda_1}.
\end{equation}
If $\lambda_1=\lambda_2$, then $\tau=1/\sqrt2$ and
$K_A=1/(2\lambda_1^2)$.

Equality in~\eqref{eq:density-main} holds if and only if $\rho=bb^*$ and
$\sigma=dd^*$, where $b,d$ have the forms in~\eqref{eq:right-vectors} with $t=\tau$
and $e_1,e_2\in\C^n$ are orthonormal eigenvectors of $A$ corresponding
to the eigenvalues $\lambda_1$ and $\lambda_2$, respectively.
For real symmetric $A$, the sharp constant for real density matrices
is also $K_A$.
\end{theorem}

\begin{proof}
Let $b,d\in\C^n$ be unit vectors, so $bb^*$ and $dd^*$ are
rank-one density matrices. Let $t=\abs{b^*d}$.
Since $bb^*dd^*-dd^*bb^*=(b^*d)bd^*-(d^*b)db^*$, expanding the
squared Frobenius norm gives
\begin{equation}\label{eq:density-pure-numerator}
 \norm{bb^*dd^*-dd^*bb^*}_{\F}^2=2t^2(1-t^2).
\end{equation}
The definition of $D(t)$ therefore gives
\begin{equation}\label{eq:density-pure-quotient}
 \frac{\norm{bb^*dd^*-dd^*bb^*}_{\F}^2}{(b^*Ab)(d^*Ad)}
 =\frac{2t^2(1-t^2)}{(b^*Ab)(d^*Ad)}
 \le\frac{2t^2(1-t^2)}{D(t)}.
\end{equation}

For arbitrary density matrices $\rho$ and $\sigma$, write their
spectral decompositions as
\[
 \rho=\sum_{i=1}^r p_i \psi_i \psi_i^*,\qquad
 \sigma=\sum_{j=1}^s q_j \chi_j \chi_j^*,
\]
where $\psi_1,\ldots,\psi_r$ are orthonormal eigenvectors of $\rho$,
$\chi_1,\ldots,\chi_s$ are orthonormal eigenvectors of $\sigma$,
$p_i,q_j>0$, and $\sum_i p_i=\sum_jq_j=1$.
By~\eqref{eq:D-first} and~\eqref{eq:D-second}, $D$ is positive
and continuous on $[0,1]$, so $2t^2(1-t^2)/D(t)$ attains its maximum.
Since $\sum_{i,j}p_iq_j=1$, convexity of the squared Frobenius norm
and~\eqref{eq:density-pure-quotient} applied to each pair $\psi_i,\chi_j$ give
\begin{align}
 \norm{\rho\sigma-\sigma\rho}_{\F}^2
 &=\left\lVert\sum_{i,j}p_iq_j
       (\psi_i \psi_i^*\chi_j\chi_j^*-\chi_j\chi_j^*\psi_i \psi_i^*)\right\rVert_{\F}^{2}\notag\\
 &\le\sum_{i,j}p_iq_j
       \norm{\psi_i \psi_i^*\chi_j\chi_j^*-\chi_j\chi_j^*\psi_i \psi_i^*}_{\F}^2\notag\\
 &\le\left(\max_{0\le t\le1}\frac{2t^2(1-t^2)}{D(t)}\right)
       \sum_{i,j}p_iq_j(\psi_i^*A\psi_i)(\chi_j^*A\chi_j)\notag\\
 &=\left(\max_{0\le t\le1}\frac{2t^2(1-t^2)}{D(t)}\right)
       \tr(A\rho)\tr(A\sigma).
 \label{eq:density-mixed-bound}
\end{align}
For the reverse bound, \cref{lem:Rayleigh} shows that equality
in~\eqref{eq:density-pure-quotient} is attained for every $t\in[0,1]$.
Since $\tr(Abb^*)=b^*Ab$ and $\tr(Add^*)=d^*Ad$,
applying~\eqref{eq:density-main} with $\rho=bb^*$ and $\sigma=dd^*$
and maximizing over all unit vectors $b,d$ gives
\[
 K_A\ge
 \max_{\norm{b}=\norm{d}=1}
 \frac{\norm{bb^*dd^*-dd^*bb^*}_{\F}^2}{(b^*Ab)(d^*Ad)}
 =\max_{0\le t\le1}\frac{2t^2(1-t^2)}{D(t)}.
\]
Together with~\eqref{eq:density-mixed-bound}, this gives
\begin{equation}\label{eq:density-scalar-reduction}
 K_A=\max_{0\le t\le1}\frac{2t^2(1-t^2)}{D(t)}.
\end{equation}

We now determine the maximum in~\eqref{eq:density-scalar-reduction}.
For $0<t<(\lambda_2-\lambda_1)/(\lambda_2+\lambda_1)$,
equation~\eqref{eq:D-first} gives
\[
 \frac{\mathrm d}{\mathrm dt}\left(\frac{2t^2(1-t^2)}{D(t)}\right)
 =\frac{4t}{\lambda_1\lambda_2}>0.
\]
For $(\lambda_2-\lambda_1)/(\lambda_2+\lambda_1)<t<1$,
equation~\eqref{eq:D-second} gives
\begin{equation}\label{eq:density-quotient-derivative}
 \frac{\mathrm d}{\mathrm dt}\left(\frac{2t^2(1-t^2)}{D(t)}\right)
 =\frac{16(\lambda_1+\lambda_2)t^4}
 {\bigl[\lambda_1(1+t)+\lambda_2(1-t)\bigr]^3}
 \biggl(\frac{\lambda_2-\lambda_1}{\lambda_2+\lambda_1}
       -\frac{2t^2-1}{t^3}\biggr).
\end{equation}
\Needspace{7\baselineskip}
Since
\[
 \frac{\mathrm d}{\mathrm dt}\left(\frac{2t^2-1}{t^3}\right)
 =\frac{3-2t^2}{t^4}>0\qquad(0<t<1),
\]
the function $(2t^2-1)/t^3$ is strictly increasing in $(0,1)$.
Its values at $t=1/\sqrt2$ and $t=1$ are $0$ and $1$, respectively.
Because $0\le(\lambda_2-\lambda_1)/(\lambda_2+\lambda_1)<1$,
there is a unique $\tau\in[1/\sqrt2,1)$ satisfying~\eqref{eq:density-tau}.
Moreover,~\eqref{eq:density-tau} gives
\begin{equation}\label{eq:density-tau-threshold}
 \tau-\frac{\lambda_2-\lambda_1}{\lambda_2+\lambda_1}
 =\frac{(1-\tau^2)^2}{\tau^3}>0.
\end{equation}
By~\eqref{eq:density-tau}, the derivative in~\eqref{eq:density-quotient-derivative}
is positive for
$(\lambda_2-\lambda_1)/(\lambda_2+\lambda_1)<t<\tau$ and negative
for $\tau<t<1$. Together with~\eqref{eq:D-first} and the continuity
of $D$, these derivative signs
show that $2t^2(1-t^2)/D(t)$ increases strictly on $[0,\tau]$ and
decreases strictly on $[\tau,1]$. Hence $\tau$ is its unique
maximizer. Substituting $t=\tau$ into~\eqref{eq:D-second}
and~\eqref{eq:density-scalar-reduction} gives~\eqref{eq:density-constant}.
If $\lambda_1=\lambda_2$, equations~\eqref{eq:density-constant}
and~\eqref{eq:density-tau} give $\tau=1/\sqrt2$ and
$K_A=1/(2\lambda_1^2)$.

Suppose equality holds in~\eqref{eq:density-main}.
By~\eqref{eq:density-scalar-reduction}, equality must then hold
in both inequalities of~\eqref{eq:density-mixed-bound}.
We first show that $\rho$ and $\sigma$ have rank one.
Because every $p_iq_j$ is positive, strict convexity of the squared
Frobenius norm shows that equality in the first inequality
of~\eqref{eq:density-mixed-bound} requires
\begin{equation}\label{eq:density-commutators-equal}
 \psi_i \psi_i^*\chi_j\chi_j^*-\chi_j\chi_j^*\psi_i \psi_i^*
 =\rho\sigma-\sigma\rho
 \qquad\textup{for every }i,j.
\end{equation}
Since every $p_iq_j$ is positive, equality in the second inequality
of~\eqref{eq:density-mixed-bound} holds for each pair $\psi_i,\chi_j$.
Equations~\eqref{eq:density-pure-quotient} and~\eqref{eq:density-scalar-reduction}
therefore give, for every $i,j$,
\[
 K_A=
 \frac{\norm{\psi_i \psi_i^*\chi_j\chi_j^*-\chi_j\chi_j^*\psi_i \psi_i^*}_{\F}^2}
 {(\psi_i^*A\psi_i)(\chi_j^*A\chi_j)}
 \le\frac{2\abs{\psi_i^*\chi_j}^2\bigl(1-\abs{\psi_i^*\chi_j}^2\bigr)}
 {D(\abs{\psi_i^*\chi_j})}\le K_A.
\]
Both inequalities are therefore equalities. Since $\tau$ is the
unique maximizer in~\eqref{eq:density-scalar-reduction}, we obtain
\begin{equation}\label{eq:density-overlaps}
 \abs{\psi_i^*\chi_j}=\tau\qquad\textup{for every }i,j.
\end{equation}
If $r\ge2$, fix $j\in\{1,\ldots,s\}$.
Since $\psi_1,\psi_2$ are orthonormal, multiplying
\eqref{eq:density-commutators-equal} with $i=1,2$
on the left by $\psi_1^*$ and on the right by $\psi_2$ gives
\[
 (\psi_1^*\chi_j)(\chi_j^*\psi_2)
 =\psi_1^*(\rho\sigma-\sigma\rho)\psi_2
 =-(\psi_1^*\chi_j)(\chi_j^*\psi_2).
\]
Thus $(\psi_1^*\chi_j)(\chi_j^*\psi_2)=0$, whereas~\eqref{eq:density-overlaps}
gives $\abs{(\psi_1^*\chi_j)(\chi_j^*\psi_2)}=\tau^2>0$.
This contradiction proves $r=1$. Interchanging $\rho$ and $\sigma$
gives $s=1$.

Since $r=s=1$, we have $p_1=q_1=1$.
Let $b=\psi_1$ and $d=\chi_1$, so $\rho=bb^*$ and $\sigma=dd^*$.
Equation~\eqref{eq:density-overlaps} gives $\abs{b^*d}=\tau$, and equality
in~\eqref{eq:density-pure-quotient} gives $(b^*Ab)(d^*Ad)=D(\tau)$.
Multiplying $d$ by a scalar of modulus one preserves $dd^*$,
so we may assume $b^*d=\tau$. If $\lambda_1<\lambda_2$,
then~\eqref{eq:density-tau-threshold} and the equality condition
in~\cref{lem:Rayleigh} show that $b,d$ have the forms
in~\eqref{eq:right-vectors} with $t=\tau$ for orthonormal
eigenvectors $e_1,e_2$ of $A$ corresponding to
$\lambda_1,\lambda_2$, respectively.
If $\lambda_1=\lambda_2$, then~\eqref{eq:D-second} gives
$(b^*Ab)(d^*Ad)=D(\tau)=\lambda_1^2$.
Since each factor is at least $\lambda_1$, both equal $\lambda_1$.
Thus $Ab=\lambda_1b$ and $Ad=\lambda_1d$. The vectors $e_1,e_2$ defined by~\eqref{eq:right-basis} with $t=\tau$
are orthonormal eigenvectors of $A$ corresponding to $\lambda_1$
and give the forms in~\eqref{eq:right-vectors} with $t=\tau$.

Conversely, let $b,d$ have the forms in~\eqref{eq:right-vectors} with
$t=\tau$, where $e_1,e_2$ are orthonormal and satisfy
$Ae_1=\lambda_1e_1$ and $Ae_2=\lambda_2e_2$.
Then
\begin{equation}\label{eq:density-attainment}
 b^*d=\tau,\qquad
 b^*Ab=d^*Ad=
 \frac{1}{2}\bigl[\lambda_1(1+\tau)+\lambda_2(1-\tau)\bigr].
\end{equation}
Equations~\eqref{eq:density-pure-numerator}, \eqref{eq:density-constant},
and~\eqref{eq:density-attainment} give
\[
 \norm{bb^*dd^*-dd^*bb^*}_{\F}^2
 =2\tau^2(1-\tau^2)=K_A(b^*Ab)(d^*Ad).
\]
Thus $\rho=bb^*$ and $\sigma=dd^*$ attain~\eqref{eq:density-main}.
For real symmetric $A$, the eigenvectors $e_1,e_2$ may be chosen
real, so real density matrices attain the same sharp constant
as in the complex case.
\end{proof}

When $A=I_n$, equations~\eqref{eq:density-constant} and~\eqref{eq:density-tau}
give $K_A=1/2$, so~\eqref{eq:density-main} recovers~\eqref{eq:density-universal}.
For $\lambda_1<\lambda_2$, the following corollary determines the
exact maximum of $\norm{\rho\sigma-\sigma\rho}_{\F}^2$ under a common
upper bound on $\tr(A\rho)$ and $\tr(A\sigma)$.

\begin{corollary}\label{cor:density-comparison}
Suppose $\lambda_1<\lambda_2$, and let $\varepsilon\ge0$.
Among all density matrices
$\rho,\sigma\in\C^{n\times n}$ satisfying
\begin{equation}\label{eq:density-expectation-constraints}
 \max\{\tr(A\rho),\tr(A\sigma)\}
 \le\lambda_1+\varepsilon(\lambda_2-\lambda_1),
\end{equation}
the maximum of $\norm{\rho\sigma-\sigma\rho}_{\F}^2$ is
\begin{equation}\label{eq:density-constrained-maximum}
 \begin{cases}
  8\varepsilon(1-\varepsilon)(1-2\varepsilon)^2,
   &\textup{if }0\le\varepsilon\le\dfrac{2-\sqrt2}{4},\\[6pt]
  \dfrac12,
   &\textup{if }\varepsilon>\dfrac{2-\sqrt2}{4}.
 \end{cases}
\end{equation}
If $0<\varepsilon<(2-\sqrt2)/4$, the maximum is attained if
and only if $\rho=bb^*$ and $\sigma=dd^*$, where $b,d$ have the
forms in~\eqref{eq:right-vectors} with $t=1-2\varepsilon$ for
orthonormal eigenvectors $e_1,e_2$ satisfying
$Ae_1=\lambda_1e_1$ and $Ae_2=\lambda_2e_2$.
\end{corollary}

\begin{proof}
Suppose first that $0<\varepsilon<(2-\sqrt2)/4$, and let
$t=1-2\varepsilon$, so $1/\sqrt2<t<1$.
To make $t$ the solution of~\eqref{eq:density-tau} for $A+sI_n$, let
\[
 s=\frac{1}{2}\biggl[
   \frac{(\lambda_2-\lambda_1)t^3}{2t^2-1}
   -\lambda_1-\lambda_2\biggr].
\]
Because
\[
 \lambda_1+s
 =\frac{(\lambda_2-\lambda_1)(1-t)(1+t-t^2)}{2(2t^2-1)}>0,
\]
the matrix $A+sI_n$ is positive definite. Its two smallest eigenvalues
are $\lambda_1+s$ and $\lambda_2+s$, and
\[
 \frac{\lambda_2-\lambda_1}{\lambda_1+\lambda_2+2s}
 =\frac{2t^2-1}{t^3}.
\]
Thus~\eqref{eq:density-tau} for $A+sI_n$ is satisfied by $\tau=t$.
Since $\tr\rho=\tr\sigma=1$ and $\varepsilon=(1-t)/2$,
condition~\eqref{eq:density-expectation-constraints} gives
\begin{align*}
 \max\bigl\{\tr\bigl((A+sI_n)\rho\bigr),
             \tr\bigl((A+sI_n)\sigma\bigr)\bigr\}
 &\le\lambda_1+s+\varepsilon(\lambda_2-\lambda_1)\\
 &=\frac{1}{2}\bigl[(\lambda_1+s)(1+t)+(\lambda_2+s)(1-t)\bigr].
\end{align*}
Applying~\eqref{eq:density-main} and~\eqref{eq:density-constant}
to $A+sI_n$, we obtain
\begin{align*}
 \norm{\rho\sigma-\sigma\rho}_{\F}^2
 &\le\frac{1}{4}K_{A+sI_n}
 \bigl[(\lambda_1+s)(1+t)+(\lambda_2+s)(1-t)\bigr]^2
 =2t^2(1-t^2)
 =8\varepsilon(1-\varepsilon)(1-2\varepsilon)^2.
\end{align*}
If $\rho,\sigma$ attain the value in~\eqref{eq:density-constrained-maximum},
equality must hold in~\eqref{eq:density-main} for $A+sI_n$.
Since $A+sI_n$ and $A$ have the same eigenspaces,
\cref{thm:density-main} gives $\rho=bb^*$ and $\sigma=dd^*$,
where $b,d$ have the forms in~\eqref{eq:right-vectors} with $t=1-2\varepsilon$
for orthonormal eigenvectors $e_1,e_2$ of $A$ corresponding to
$\lambda_1,\lambda_2$, respectively.

Conversely, let $b,d$ have the forms in~\eqref{eq:right-vectors}
with $t=1-2\varepsilon$ for orthonormal eigenvectors $e_1,e_2$
corresponding to $\lambda_1,\lambda_2$, respectively.
Then $\rho=bb^*$ and $\sigma=dd^*$ satisfy
\[
 \tr(A\rho)=\tr(A\sigma)
 =\frac{1}{2}\bigl[\lambda_1(1+t)+\lambda_2(1-t)\bigr]
 =\lambda_1+\varepsilon(\lambda_2-\lambda_1),
\]
and~\eqref{eq:density-pure-numerator} gives
$\norm{\rho\sigma-\sigma\rho}_{\F}^2=2t^2(1-t^2)$.
These pairs therefore attain the maximum.

If $\varepsilon=0$, both expectation values equal $\lambda_1$.
Because $\lambda_1<\lambda_2$, both density matrices must equal
$e_1e_1^*$, where $e_1$ is a unit eigenvector corresponding to
$\lambda_1$. The maximum is therefore zero.
If $\varepsilon\ge(2-\sqrt2)/4$, equation~\eqref{eq:density-universal}
gives the upper bound $1/2$. The vectors in~\eqref{eq:right-vectors}
with $t=1/\sqrt2$ and $Ae_i=\lambda_i e_i$ satisfy
\[
 b^*Ab=d^*Ad
 =\lambda_1+\frac{2-\sqrt2}{4}(\lambda_2-\lambda_1)
 \le\lambda_1+\varepsilon(\lambda_2-\lambda_1).
\]
By~\eqref{eq:density-pure-numerator}, the density matrices $bb^*,dd^*$
attain $1/2$, proving~\eqref{eq:density-constrained-maximum}.
\end{proof}

Normalizing by the traces gives an equivalent form
of~\eqref{eq:density-main} for positive semidefinite matrices.

\begin{corollary}\label{cor:density-psd}
For all positive semidefinite $R,S\in\C^{n\times n}$,
\[
 \norm{RS-SR}_{\F}^2
 \le K_A\tr(R)\tr(AR)\tr(S)\tr(AS).
\]
The constant $K_A$ is sharp. For nonzero $R,S$, equality holds if
and only if $R=rbb^*$ and $S=sdd^*$, where $r,s>0$ and $b,d$ have
the forms in~\eqref{eq:right-vectors} with $t=\tau$ for orthonormal
eigenvectors $e_1,e_2$ satisfying $Ae_1=\lambda_1e_1$ and
$Ae_2=\lambda_2e_2$.
\end{corollary}

\begin{proof}
If $R=0$ or $S=0$, both sides are zero. Otherwise,
$R/\tr(R)$ and $S/\tr(S)$ are density matrices, so
\eqref{eq:density-main} gives
\[
 \frac{\norm{RS-SR}_{\F}^2}{\tr(R)^2\tr(S)^2}
 \le K_A\frac{\tr(AR)}{\tr(R)}\frac{\tr(AS)}{\tr(S)}.
\]
Multiplying by $\tr(R)^2\tr(S)^2$ proves the inequality.
The equality characterization and sharpness follow from
\cref{thm:density-main}.
\end{proof}


\begin{thebibliography}{99}

\bibitem{BW}
A.~B\"ottcher and D.~Wenzel,
The Frobenius norm and the commutator,
\emph{Linear Algebra Appl.} \textbf{429} (2008), 1864--1885.
\href{https://doi.org/10.1016/j.laa.2008.05.020}{doi:10.1016/j.laa.2008.05.020}.

\bibitem{CAJ}
C.-M.~Cheng, D.~O.~Akintoye, and R.~Jiao,
Commutator bounds and region of singular values of the commutator
with a rank one matrix,
\emph{Linear Algebra Appl.} \textbf{613} (2021), 347--376.
\href{https://doi.org/10.1016/j.laa.2020.11.013}{doi:10.1016/j.laa.2020.11.013}.

\bibitem{CJ}
C.-M.~Cheng and R.~Jiao,
Proof of Wenzel's conjecture concerning singular values of the
commutator of rank one matrices,
\emph{Linear Algebra Appl.} \textbf{592} (2020), 165--174.
\href{https://doi.org/10.1016/j.laa.2020.01.028}{doi:10.1016/j.laa.2020.01.028}.

\bibitem{CL}
C.-M.~Cheng and Y.~Liang,
Singular values, eigenvalues and diagonal elements of the commutator
of $2\times2$ rank one matrices,
\emph{Electron. J. Linear Algebra} \textbf{36} (2020), 1--20.
\href{https://doi.org/10.13001/ela.2020.4955}{doi:10.13001/ela.2020.4955}.

\bibitem{FC}
W.~Fang and C.-M.~Cheng,
On some conjectures concerning the B\"ottcher--Wenzel inequality
for weighted Frobenius norms,
\emph{Linear Algebra Appl.} \textbf{718} (2025), 1--13.
\href{https://doi.org/10.1016/j.laa.2025.03.015}{doi:10.1016/j.laa.2025.03.015}.

\bibitem{ICS}
P.~Iyengar, G.~N.~Chandan, and R.~Srikanth,
Quantifying quantumness via commutators: an application to quantum walks,
arXiv:1312.1329, 2013.
\href{https://arxiv.org/abs/1312.1329}{arxiv.org/abs/1312.1329}.

\bibitem{JWC}
J.~Jing, L.-A.~Wu, and A.~del Campo,
Fundamental speed limits to the generation of quantumness,
\emph{Scientific Reports} \textbf{6} (2016), 38149.
\href{https://doi.org/10.1038/srep38149}{doi:10.1038/srep38149}.

\bibitem{KMOLC}
G.~Kimura, A.~Mayumi, H.~Ohno, J.~Lee, and D.~Chru\'sci\'nski,
Beyond Robertson--Schr\"odinger: A general uncertainty relation unveiling
hidden noncommutative trade-offs,
\emph{npj Quantum Inf.} (2026).
\href{https://doi.org/10.1038/s41534-026-01374-0}{doi:10.1038/s41534-026-01374-0}.

\bibitem{MKOC}
A.~Mayumi, G.~Kimura, H.~Ohno, and D.~Chru\'sci\'nski,
B\"ottcher--Wenzel inequality for weighted Frobenius norms and its
application to quantum physics,
\emph{Linear Algebra Appl.} \textbf{700} (2024), 35--49.
\href{https://doi.org/10.1016/j.laa.2024.07.013}{doi:10.1016/j.laa.2024.07.013}.

\bibitem{RD}
B.~Russo and H.~A.~Dye,
A note on unitary operators in $C^*$-algebras,
\emph{Duke Math. J.} \textbf{33} (1966), 413--416.
\href{https://doi.org/10.1215/S0012-7094-66-03346-1}{doi:10.1215/S0012-7094-66-03346-1}.

\bibitem{VJ}
S.~Vong and X.~Jin,
Proof of B\"ottcher and Wenzel's conjecture,
\emph{Oper. Matrices} \textbf{2} (2008), 435--442.
\href{https://doi.org/10.7153/oam-02-26}{doi:10.7153/oam-02-26}.

\bibitem{Wenzel}
D.~Wenzel,
A strange phenomenon for the singular values of commutators with
rank one matrices,
\emph{Electron. J. Linear Algebra} \textbf{30} (2015), 649--669.
\href{https://doi.org/10.13001/1081-3810.3135}{doi:10.13001/1081-3810.3135}.

\end{thebibliography}
\end{document}